\documentclass[10pt]{article}
\font\smallit=cmti10
\font\smalltt=cmtt10

\usepackage{amssymb,latexsym,amsmath,epsfig,amsthm} 

\makeatletter

\renewcommand\section{\@startsection {section}{1}{\z@}
{-30pt \@plus -1ex \@minus -.2ex}
{2.3ex \@plus.2ex}
{\normalfont\normalsize\bfseries}}

\renewcommand\subsection{\@startsection{subsection}{2}{\z@}
{-3.25ex\@plus -1ex \@minus -.2ex}
{1.5ex \@plus .2ex}
{\normalfont\normalsize\bfseries}}

\renewcommand{\@seccntformat}[1]{\csname the#1\endcsname. }

\makeatother

\newtheorem{theorem}{Theorem}
\newtheorem{lemma}{Lemma}

\newtheorem{corollary}{Corollary}

\newcommand{\beq}{\begin{equation}}
\newcommand{\eeq}{\end{equation}}

\def\a{\alpha}

\def\({\left(}
\def\){\right)}

\begin{document}

\begin{center}
\uppercase{\bf A method for obtaining Fibonacci identities}
\vskip 20pt
{\bf Dmitry I. Khomovsky}\\
{\smallit Lomonosov Moscow State University}\\
{\tt khomovskij@physics.msu.ru}
\end{center}
\vskip 30pt

\centerline{\smallit Received: , Revised: , Accepted: , Published: } 
\vskip 30pt

\centerline{\bf Abstract}

\noindent
For the Lucas sequence $\{U_{k}(P,Q)\}$ we discuss the identities such as the well-known \mbox{Fibonacci} identities. We also propose a method for obtaining identities involving recurrence sequences, with the help of which we find an interpolating-type identity for second-order linear recurrence sequences.

\pagestyle{myheadings}
\markright{\smalltt INTEGERS: 18 (2018)\hfill}
\thispagestyle{empty}
\baselineskip=12.875pt
\vskip 30pt

\section{Introduction}
Let $\mathbb{F}$ be an arbitrary field and $P, Q$ be its nonzero elements.  The Lucas sequences $\{U_{k}(P,Q)\}$, $\{V_{k}(P,Q)\}$ are defined recursively by
\beq \label{DefLuc}
f_{k+2}=P f_{k+1}-Q f_{k},
\eeq
with the initial values\footnote{In this paper instead of $U_{k}(P,Q)$ and $V_{k}(P,Q)$ we will write $U_k$ and $V_k$, if it is not ambiguous.} $U_{0}=0,\ U_{1}=1,\, V_{0}=2,\, V_{1}=P$. The characteristic equation of the recurrence relation $(\ref{DefLuc})$ is $x^2-P x+Q=0$. Its roots are  $\a=\frac{P+\sqrt{\Delta}}{2}$ and $\bar{\a}=\frac{P-\sqrt{\Delta}}{2}$, where $\Delta=P^2-4Q$. If $\Delta\not=0$, the Lucas sequences can be expressed in terms of $\a$ and $\bar{\a}$ according to Binet formulas
\beq \label{BF}
U_{k}=\frac{\a^k-\bar{\a}^k}{\a-\bar{\a}},\,\,\,V_{k}=\a^k+\bar{\a}^k,\,\,\, P=\a+\bar{\a},\,\,\, Q=\a\bar{\a},\,\,\,\sqrt{\Delta}=\a-\bar{\a}.
\eeq
We define $U_{-k}=-U_{k}/Q^k$ and $V_{-k}=V_{k}/Q^k$ for $k\geq1$. Then Binet formulas
are valid for every $k\in \mathbb{Z}$. More information about the subject can be found in \cite{-2,-1,13}.

The sequence   of the Fibonacci numbers $\{F_{k}\}$   is defined by the recurrence relation $F_{k+2}=F_{k+1}+F_{k}$ $(F_0=0$, $F_1=1)$ \cite{1}. From this definition, it follows that $F_k=U_{k}(1,-1)$. There are many identities involving the Fibonacci numbers \cite{-3,-4}.  Some of them are derived from identities involving $\{U_{k}\}$ if we put $P=1$, $Q=-1$. In this paper, we generalize some identities for $\{F_{k}\}$ in terms of $\{U_{k}\}$. Often  such generalized identities have a form close to the initial one. We also present a method for obtaining identities involving recurrence sequences. To show the efficiency of this method we obtain some identities for the Fibonacci numbers. The most interesting result is presented in Theorem 2.

\section{\bf{Generalizations of Fibonacci identities}}
It is clear that any identity which holds  for $\{U_{k}\}$ can be easy transformed into an identity for $\{F_{k}\}$. But there exist  identities involving $\{F_{k}\}$ for which  analogues in terms of $\{U_{k}\}$ are unknown. We note that the discussion fits in the literature \cite{1.10.1, 1.4}. For example,  Candido's identity \cite{1.3} was generalized in \cite{1.1}. This result is as follows:

\smallskip
\noindent$(F.1)$\phantom{$(GF)$}$2(F_{k}^4+F_{k+1}^4+F_{k+2}^4)=(F_{k}^2+F_{k+1}^2+F_{k+2}^2)^2$,

\smallskip
\noindent$(GF.1)$\phantom{$(F)$}$2(Q^4 U_{k}^4+P^4 U_{k+1}^4+U_{k+2}^4)=(Q^2 U_{k}^2+P^2 U_{k+1}^2+U_{k+2}^2)^2$.

\smallskip
\noindent Catalan's identity and its generalization \cite{-3} are:

\smallskip
\noindent$(F.2)$\phantom{$(GF)$}$F_{k}^2-F_{k+n}F_{k-n}=(-1)^{k-n}F_{n}^2$,

\smallskip
\noindent$(GF.2)$\phantom{$(F)$}$U_{k}^2-U_{k+n}U_{k-n}=Q^{k-n}U_{n}^2$.

\smallskip

\subsection{Two important identities for $\{U_{k}(P,Q)\}$}
For the Fibonacci numbers  the following holds:

\smallskip
\noindent$(F.3)$\phantom{$(GF)$}$F_{2k+1}=F_{k+1}^2+F_{k}^2$,

\smallskip
\noindent$(F.4)$\phantom{$(GF)$}$F_{2k}=F_{k}(2 F_{k+1}-F_{k})$.

\smallskip
\noindent The generalizations of $(F.3)$, $(F.4)$ are:

\smallskip
\noindent$(GF.3)$\phantom{$(F)$}$U_{2k+1}=U_{k+1}^2-Q U_{k}^2$,

\smallskip
\noindent$(GF.4)$\phantom{$(F)$}$U_{2k}=U_{k}(2 U_{k+1}-P U_{k})$.
\begin{lemma}\label{L1}
Let $\{U_{k}\}$ be the Lucas sequence with parameters $P, Q$ in an arbitrary field $\mathbb{F}$. Then $(GF.3)$ and $(GF.4)$ are valid.\end{lemma}
\begin{proof} Consider the well-known identity $U_{n+m}=U_{n}U_{m+1}-Q U_{m}U_{n-1}$. If we put $n=k+1$ and $m=k$, then we obtain $(GF.3)$. If we put $n=k$ and $m=k$, then we obtain $U_{2k}=U_{k}(U_{k+1}-Q U_{k-1})$. If we use  $Q U_{k-1}=P U_{k}-U_{k+1}$ in the previous equality, then we obtain $(GF.4)$.
\end{proof}
These two identities can be found in the literature: both identities are presented in \cite{1.5},  $(GF.3)$ is in \cite{1.7, 1.6}.
To show their importance, we note that $(GF.3)$, $(GF.4)$ can be used to calculate the Lucas sequences.
This is an alternative to the more known method \cite{1.8}, which uses the properties of both sequences $\{U_{k}\}$ and $\{V_{k}\}$.
\subsection{Higher order Fibonacci identities}
The following identities hold:

\smallskip
\noindent$(F.5)$\phantom{$(GF)$}$F_{3k}=F_{k+1}^3+F_{k}^3-F_{k-1}^3$,

\smallskip
\noindent$(F.6)$\phantom{$(GF)$}$F_{4k}=F_{k+1}^4+2F_{k}^4-F_{k-1}^4+4F_{k}^3 F_{k-1}$,

\smallskip
\noindent$(F.7)$\phantom{$(GF)$}$F_{5k}=F_{k+1}^5+4F_{k}^5-F_{k-1}^5+10F_{k+1}F_{k}^3 F_{k-1}$.

\smallskip
\noindent The reader can find $(F.5)$, $(F.6)$ in \cite{1.9,1.10}. The generalizations of the above are:

\smallskip
\noindent$(GF.5)$\phantom{$(F)$}$U_{3k}=U_{k+1}^3/P+(U_3-P^2) U_{k}^3+Q^3 U_{k-1}^3/P$,

\smallskip
\noindent$(GF.6)$\phantom{$(F)$}$U_{4k}=U_{k+1}^4/P+(U_4-P^3)U_{k}^4-Q^4 U_{k-1}^4/P+4Q^2U_{k}^3 U_{k-1}$,

\smallskip
\noindent$(GF.7)$\phantom{$(F)$}$U_{5k}=U_{k+1}^5/P+(U_5-P^4)U_{k}^5+Q^5 U_{k-1}^5/P+10Q^2 U_{k+1} U_{k}^3 U_{k-1}$.

\smallskip
\noindent One way to prove  $(GF.5)$-$(GF.7)$ is to use $(\ref{BF})$. But obtaining such identities by using the Binet formulas is difficult. In the next section we discuss  how we obtain similar identities.
\section{A new method for obtaining identities for linear recurrences}
First we consider the matrix method that is often used to prove some identities concerning the generalized Fibonacci and Lucas numbers \cite{1.7,1.5}. We have the following matrix formulas:
\beq\label{M0}
\begin{pmatrix}
U_{k+1} & V_{k+1} \\
U_{k} & V_{k}
\end{pmatrix}=M^{k}\begin{pmatrix}
1 & P \\
0 & 2
\end{pmatrix},\,\,\,\,\,\,\,\,\text{where} \,\,
M=\begin{pmatrix}
P & -Q \\
1 & 0
\end{pmatrix},
\eeq
\beq\label{M}
\begin{pmatrix}
U_{k} & U_{k+1} \\
V_{k} & V_{k+1}
\end{pmatrix}=\begin{pmatrix}
0 & 1 \\
2 & P
\end{pmatrix}{R}^{k},\,\,\,\,\,\,\,\,\,\, \text{where}\,\,
R=\begin{pmatrix}
0 & -Q \\
1 & P
\end{pmatrix}.
\eeq
\begin{theorem}\label{T1}
Let $\{U_{r}(P,Q)\}$, $\{V_{r}(P,Q)\}$  be the Lucas sequences with $P, Q\in\mathbb{F}$. Then there exist the following representations of $U_{mk}$ and $V_{mk}$ via $U_k$, $U_{k+1}$:
\beq\label{Rep}
U_{mk}=\sum_{i=0}^{m}\binom{m}{i}(-1)^{i+1}U_i U_k^i U_{k+1}^{m-i},
\eeq
\beq\label{Rep1}
V_{mk}=\sum_{i=0}^{m}\binom{m}{i}(-1)^{i}V_i U_k^i U_{k+1}^{m-i},
\eeq
where $m\in \mathbb{Z^+}$, $k\in\mathbb{Z}$ and such that $U_k U_{k+1}\not =0$.
Moreover, there are no other representations of the form $\sum_{i=0}^{m} c_i U_k^i U_{k+1}^{m-i}$, where $c_i$ are coefficients that do not depend on $k$.\end{theorem}
\begin{proof}
By $(\ref{M0})$, we have
\beq\label{fM0}
M^{k}=
\begin{pmatrix}
U_{k+1} & (-P U_{k+1}+V_{k+1})/2 \\
U_{k} & (-P U_{k}+V_{k})/2
\end{pmatrix}.
\eeq
It is well known that $V_k=P U_k-2 Q U_{k-1}$. Then
\beq\label{fM1}
M^{k}=
\begin{pmatrix}
U_{k+1} & -Q U_{k} \\
U_{k} & -Q U_{k-1}
\end{pmatrix}=\begin{pmatrix}
U_{k+1} & -Q U_{k} \\
U_{k} & U_{k+1}-P U_{k}
\end{pmatrix}
\eeq
since $U_{k+1}=P U_k-Q U_{k-1}$ and therefore
\beq\label{fM2}
M^k=U_{k+1}\begin{pmatrix}
1 & 0\\
0 & 1
\end{pmatrix}+U_k\begin{pmatrix}
0 & -Q\\
1 & -P
\end{pmatrix}=U_{k+1}I+U_k A,
\eeq
where
\beq\label{fM3}
A=\begin{pmatrix}
0 & -Q\\
1 & -P
\end{pmatrix}.\eeq
It follows from $(\ref{M})$  that if $P$ is replaced by $-P$ in $R$, then we obtain $A$.
In addition, if we use $V_k=P U_k-2 Q U_{k-1}$ in the matrix formula $(\ref{M})$, then
\beq\label{fM4}
R^{m}=
\begin{pmatrix}
-Q U_{m-1}(P,Q) & -Q U_m(P,Q) \\
U_{m}(P,Q) & U_{m+1}(P,Q)
\end{pmatrix}.
\eeq
By the Binet formulas, it can be seen that $U_m(-P,Q)=(-1)^{m+1}U_m(P,Q)$. Then
\beq\label{fM5}
A^{m}=\begin{pmatrix}
(-1)^{m+1}Q U_{m-1}(P,Q) & (-1)^{m}Q U_m(P,Q) \\
(-1)^{m+1}U_{m}(P,Q) & (-1)^{m}U_{m+1}(P,Q)
\end{pmatrix}.
\eeq
Thus
\beq\label{fM6}
M^{mk}=(M^k)^m=(U_{k+1}I+U_k A)^m=\sum_{i=0}^{m}\binom{m}{i}U_k^i A^i U_{k+1}^{m-i}.\eeq
Since
\beq\label{fM7}
M^{mk}=\begin{pmatrix}
* & * \\
U_{mk} & *
\end{pmatrix}\,\,\, \text{ and }\,\, A^{i}=\begin{pmatrix}
* & * \\
(-1)^{i+1}U_{i} & *
\end{pmatrix},
\eeq
the representation $(\ref{Rep})$ follows from $(\ref{fM6})$.
Let $\operatorname{tr}(X)$ denote trace of the matrix $X$. From $(\ref{fM1})$ and $(\ref{fM6})$, it is seen that
\begin{align}\label{fM8}
\operatorname{tr}(M^{mk})&=U_{mk+1}-Q U_{mk-1}=V_{mk}=\operatorname{tr}\left(\sum_{i=0}^{m}\binom{m}{i}U_k^i A^i U_{k+1}^{m-i}\right)\nonumber\\
&=\sum_{i=0}^{m}\binom{m}{i}U_k^i \operatorname{tr}\left(A^i\right) U_{k+1}^{m-i}.\end{align}
By $(\ref{fM5})$, $\operatorname{tr}\left(A^i\right)=(-1)^i(-QU_{i-1}+U_{i+1})=(-1)^i V_i$. Using this in $(\ref{fM8})$, we get $(\ref{Rep1})$.

Let $a_i(P,Q)=\binom{m}{i}(-1)^{i+1}U_i$, then $U_{mk}=\sum_{i=0}^{m}a_i(P,Q)U_k^i U_{k+1}^{m-i}.$ Suppose that there is another representation in the analogical form $U_{mk}=\sum_{i=0}^{m}b_i(P,Q)U_k^i U_{k+1}^{m-i}.$
Then $\sum_{i=0}^{m}(b_i-a_i)U_k^i U_{k+1}^{m-i}=0$. We can take $P=2, Q=1$; in this case $U_k=k$. Therefore, we have $\sum_{i=0}^{m}(b_i-a_i)k^i (k+1)^{m-i}\equiv0\pmod{ k}$, and hence $k\mid (b_0-a_0)$. This must hold for any $k>0$, and  is possible only if $b_0=a_0$. By analogy we prove that $b_i=a_i$ for $0< i\leq m$. So the representation $(\ref{Rep})$ is unique. The uniqueness of the representation $(\ref{Rep1})$ can be proved similarly.
\end{proof}
We note that the study fits in the literature. For a generalization of $(\ref{Rep})$ see Remark $4.1$ in \cite{00}. When $P^2-4Q$ is not a perfect square (see \cite{0}, Corollary $2.10$), both identities are proved.

Now we consider another way of obtaining the identity $(\ref{Rep})$.
We can use a method similar to the partial fraction decomposition. We begin with the supposed identity $U_{mk}=\sum_{i=0}^{m} a_i(P,Q) U_k^i U_{k+1}^{m-i}$.  Then we get the system of $m+1$ equations whose variables and coefficients are functions of $P, Q$:
\beq\label{System}
U_{mk}=\sum_{i=0}^{m}a_i(P,Q)U_k^i U_{k+1}^{m-i}\,\,\,\, \left( -\left\lfloor (m-1)/2\right\rfloor\leq k\leq \left\lceil m/2\right\rceil \right).
\eeq
Since $U_0=0$, $U_1=1$, it is seen that $a_0=0$, $a_m=(-1)^{m+1}U_m$. The remaining coefficients can be found by solving the system.

\noindent{\bf Example 1.} $U_{3k}=a_0 U_{k+1}^3+a_1 U_{k+1}^2 U_{k}+a_2 U_{k+1} U_{k}^2+a_3 U_{k}^3$. The system $(\ref{System})$ for this case is:
\beq\begin{cases}
U_{-3}=a_0 U_{0}^3+a_1 U_{0}^2 U_{-1}+a_2 U_{0} U_{-1}^2+a_3 U_{-1}^3,\\
U_{0}=a_0 U_{1}^3+a_1 U_{1}^2 U_{0}+a_2 U_{1} U_{0}^2+a_3 U_{0}^3,\\
U_{3}=a_0 U_{2}^3+a_1 U_{2}^2 U_{1}+a_2 U_{2} U_{1}^2+a_3 U_{1}^3,\\
U_{6}=a_0 U_{3}^3+a_1 U_{3}^2 U_{2}+a_2 U_{3} U_{2}^2+a_3 U_{2}^3.
\end{cases}\eeq
Using $U_{-3}=-(P^2-Q)/Q^3$, $U_{-1}=-1/Q$, $U_0=0$, $U_{1}=1$, $U_{3}=P^2-Q$, $U_{6}=P^5-4P^3Q+3P Q^2$, we get the solution $a_0=0$, $a_1=3$, $a_2=-3P$, $a_3=P^2-Q$. So $U_{3k}=3 U_{k+1}^2 U_{k}-3P U_{k+1} U_{k}^2+(P^2-Q) U_{k}^3.$ This is consistent with  $(\ref{Rep})$.

\noindent{\bf Example 2.} Consider the  Fibonacci Pythagorean triples identity \cite{1.12}.

\smallskip
\noindent$(F.8)$\phantom{$(GF)$}$(F_{k-1}F_{k+2})^2+(2F_{k}F_{k+1})^2=F_{2k+1}^2$.

\smallskip
\noindent
Suppose that there exists an identity of the form
\beq\label{Sup1}c_1(U_{k-1}U_{k+2})^2+c_2(U_{k}U_{k+1})^2+c_3 U_{2k+1}^2=0.\eeq
To get the system for the variables $c_i$ we put $k=-1, 0, 1$. Then
\beq\label{Syst2}\begin{cases}
c_1 P^2/Q^4+c_3/Q^2=0,\\
c_1 P^2/Q^2+c_3=0,\\
c_2 P^2+c_3(P^2-Q)^2=0.
\end{cases}\eeq
If the system has no solutions (the rank is $3$), then we can state that the identity of the form $(\ref{Sup1})$ does not exist. In this case we may modify it by adding new terms and obtain a new system.
In fact, the rank is $2$. We put $c_3=-1$, then the solution is $c_1=Q^2/P^2, c_2=(P^2-Q)^2/P^2, c_3=-1$. But when we use the Binet formulas to verify the resulting formula  we see that it is not valid. Moreover, if we get the system which corresponds  $k=0, 1, 2$, then the determinant of the system matrix is $-2 P^4 \left(P^2-2 Q\right) \left(P^2-Q\right) \left(P^2+Q\right)/Q$. Thus, the solution exists if one of the following holds: $P=0, P^2=2 Q, P^2=\pm Q$. But if we add $c_4U_{k-1}U_{k}U_{k+1}U_{k+2}$ to the left side of $(\ref{Sup1})$, then we find the generalization of $(F.8)$ as follows:

\smallskip
\noindent$(GF.8)$\phantom{$(F)$}$(Q U_{k-1}U_{k+2})^2+((P^2-Q)U_{k}U_{k+1})^2=(P U_{2k+1})^2+$

\smallskip
\par \hfill \noindent\phantom{$(GF.8)(F)(Q U_{k-1}U_{k+2})^2+((P^2-Q)U_{k}U_{k+1})^2=$}\,$2Q (P^2+Q)U_{k-1}U_{k}U_{k+1}U_{k+2}$.

\smallskip
\noindent{\bf Remark.} In general, we cannot assert that the method leads to the final result which holds for all $k$ since a supposed formula similar to $(\ref{Sup1})$  is checked by a system only for some values of $k$. So we need to use  the Binet formulas  to prove that the final result is valid for all $k$. In Example $1$ this check is not necessary, since by Theorem \ref{T1} we know that the identity which involves $U_{mk}$, $U_{k}$, $U_{k+1}$ exists. But we obtained the unique solution using the method.

\noindent{\bf Example 3.} We want to get an identity of the form
\beq\label{Sup2}c_1 U_{k+1}^2+c_2U_{k}^2+c_3 U_{k-1}^2=0.\eeq
We put $k=-1, 0, 1$ and obtain the following system
\beq\label{System2}\begin{cases}
c_2 /Q^2+c_3 P^2/Q^4=0,\\
c_1+c_3/Q^2=0,\\
c_1 P^2+c_2=0.
\end{cases}\eeq
Since the determinant of the system matrix is $2P^2/Q^4$, we conclude that for a nonzero $P$ there is no identity which contains only squares of three consecutive terms of $\{U_k\}$. But if we try to find an identity of the form
\beq\label{Sup3}c_1 U_{k+1}^2+c_2U_{k}^2+c_3 U_{k-1}^2+c_4 U_{k-2}^2=0,\eeq
then we obtain

\smallskip
\noindent$(GF.9)$\phantom{$(F)$}$U_{k+1}^2-Q^3 U_{k-2}^2=(P^2-Q)(U_{k}^2-Q U_{k-1}^2)$.

\smallskip
\noindent If $P=1$ and $Q=-1$, then we get the identity  for the Fibonacci numbers:

\smallskip
\noindent$(F.9)$\phantom{$(GF)$}$F_{k+1}^2+F_{k-2}^2=2(F_{k}^2+F_{k-1}^2)$.

\smallskip
As shown the method is good not only for generalizing  Fibonacci identities, but also for finding new identities. Below we present some results that we obtained using this method. Some of them are well-known.

\smallskip
\noindent$(F.10)$\phantom{$(GF$}$F_{2k}=F_{k+1}^2- F_{k-1}^2$,

\smallskip
\noindent$(GF.10)$\phantom{$(F$}$P U_{2k}=U_{k+1}^2- Q^2 U_{k-1}^2,$

\smallskip
\noindent$(F.11)$\phantom{$(GF$}$F_{k+1}^2=4F_{k}F_{k-1}+F_{k-2}^2$,

\smallskip
\noindent$(GF.11)$\phantom{$(F$}$U_{k+1}^2=2P(P^2-Q)U_k U_{k-1}+(Q^2-P^4) U_{k-1}^2+P^2Q^2U_{k-2}^2$,

\smallskip
\noindent$(F.12)$\phantom{$(GF$}$F_{k+2}^2-F_{k-2}^2=3(F_{k+1}^2-F_{k-1}^2)$,

\smallskip
\noindent$(GF.12)$\phantom{$(F$}$U_{k+2}^2-Q^4U_{k-2}^2=(P^2-2Q)(U_{k+1}^2-Q^2U_{k-1}^2)$.

\smallskip
\noindent
Note that $(GF.9)$ and $(GF.12)$ have similar forms. To generalize $(GF.12)$ consider $c_1 U_{k+m}^2+c_2 U_{k+l}^2+c_3U_{k-l}^2+c_4U_{k-m}^2=0$. With the help of the method we get the following system:
\beq\label{System2.1}\begin{cases}
c_1 U_{m-1}^2+c_2 U_{l-1}^2+c_3 U_{l+1}^2/Q^{2(l+1)}+c_4 U_{m+1}^2/Q^{2(m+1)}=0,\\
c_1 U_{m}^2+c_2 U_{l}^2+c_3 U_{l}^2/Q^{2l}+c_4 U_{m}^2/Q^{m}=0,\\
c_1 U_{m+1}^2+c_2 U_{l+1}^2+c_3 U_{l-1}^2/Q^{2(l-1)}+c_4 U_{m-1}^2/Q^{2(m-1)}=0.
\end{cases}\eeq
The rank is $2$. If we put $c_1=U_{l+1}^2-Q^2 U_{l-1}^2$, then $c_2=-(U_{m+1}^2-Q^2 U_{m-1}^2)$, $c_3=Q^{2l}(U_{m+1}^2-Q^2 U_{m-1}^2)$, $c_4=-Q^{2m}(U_{l+1}^2-Q^2 U_{l-1}^2)$. Using $(GF.10)$, we obtain

\smallskip
\noindent$(GF.13)$\phantom{$(F$}$U_{2l}(U_{k+m}^2-Q^{2m}U_{k-m}^2)=U_{2m}(U_{k+l}^2-Q^{2l}U_{k-l}^2)$,

\smallskip
\noindent$(F.13)$\phantom{$(GF$}$F_{2l}(F_{k+m}^2-F_{k-m}^2)=F_{2m}(F_{k+l}^2-F_{k-l}^2)$.

\smallskip
\noindent Another way to prove $(GF.13)$ is to  use Catalan's identity $(GF.2)$. It follows that $U_{k+m}^2-Q^{2m}U_{k-m}^2=U_{2k}U_{2m}$, $U_{k+l}^2-Q^{2l}U_{k-l}^2=U_{2k}U_{2l}$, which completes the proof.

To generalize  $(GF.9)$ we consider the most general formula which involves  only four squares of sequence terms: $U_{k}^2+c_1 U_{k+m}^2+c_2U_{k+l}^2+c_3U_{k+n}^2=0$. Using the method, we get

\smallskip
\noindent$(GF.14)$\phantom{$(F$}$U_{k}^2=U_{k+m}^2U_{l}U_{s}/(Q^{2m}U_{l-m}U_{s-m})+U_{k+l}^2U_{s}U_{m}/(Q^{2l}U_{s-l}U_{m-l})+$

\smallskip
\par \hfill \noindent\phantom{$(GF.12)(F$}$U_{k+s}^2U_{m}U_{l}/(Q^{2s}U_{m-s}U_{l-s})$,

\smallskip
\noindent$(F.14)$\phantom{$(GF$}$F_{k}^2=F_{k+m}^2F_{l}F_{s}/(F_{l-m}F_{s-m})+F_{k+l}^2F_{s}F_{m}/(F_{s-l}F_{m-l})+$

\smallskip
\par \hfill \noindent\phantom{$(GF.14)(F$}$F_{k+s}^2F_{m}F_{l}/(F_{m-s}F_{l-s})$.

\smallskip
\noindent Here we mean that $l, m, s$ are  distinct integers. The analog for cubes is

\smallskip
\noindent$(GF.15)$\phantom{$(F$}$U_{k}^3=U_{k+m}^3U_{l}U_{p}U_{s}/(Q^{3m}U_{l-m}U_{p-m}U_{s-m})+$

\smallskip
\par \hfil \noindent\phantom{$FF)$}$U_{k+l}^3U_{m}U_{p}U_{s}/(Q^{3l}U_{m-l}U_{p-l}U_{s-l})+$\hfil

\smallskip
\par \hfil \noindent\phantom{$(GF.12.GFFFFFFF)(F$}$U_{k+p}^3U_{l}U_{m}U_{s}/(Q^{3p}U_{l-p}U_{m-p}U_{s-p})+$\hfil

\smallskip
\par \hfill \noindent\phantom{$(GF.12)(F$}$U_{k+s}^3U_{l}U_{m}U_{p}/(Q^{3s}U_{l-s}U_{m-s}U_{p-s}).$

\smallskip
\noindent Similar identities on sums of powers of the terms of the Lucas sequences can be found in \cite{2000,1999}.
\begin{theorem}\label{T2}
Let $\{U_{k}\}$ be the Lucas sequence with parameters $P, Q$ in an arbitrary field $\mathbb{F}$. Let $d_i$ $(0\leq i\leq n)$ be distinct integers. Then the following holds:
\beq \label{Lagrange}
U_{k+x}^n=\sum_{i=0}^{n}U_{k+d_i}^n\prod_{\substack{j=0\\ j\not=i}}^{n}\frac{U_{x-d_j}}{U_{d_i-d_j}}.
\eeq
\end{theorem}
\begin{proof}
We denote $(\alpha/\bar{\alpha})^x=y$. Then
\beq\label{Lag1} U_{k+x}^n=\bar{\alpha}^{x n}\left(\frac{\alpha^{k+x}/\bar{\alpha}^{x}-\bar{\alpha}^{k}}{\alpha-\bar{\alpha}}\right)^n=\bar{\alpha}^{x n}\left(\frac{\alpha^{k} y-\bar{\alpha}^{k}}{\alpha-\bar{\alpha}}\right)^n.\eeq
We see that $U_{k+x}^n$ is the product of $\bar{\alpha}^{x n}$ and the polynomial in $y$ of degree $n$ with coefficients in $\mathbb{F}$. Since $U_{x-d_j}=\bar{\alpha}^{x}(\alpha^{-d_j}y-\bar{\alpha}^{-d_j})/(\alpha-\bar{\alpha})$, the right side of $(\ref{Lagrange})$ has the same structure as $U_{k+x}^n$.
Since $U_0=0$, it easy to see that $(\ref{Lagrange})$ is valid for $x=d_i$ $(0\leq i\leq n)$.
Thus, the polynomials of degree $n$ on the left and right sides of $(\ref{Lagrange})$ are equal, since they are equal for $n+1$ different values of the variable.
\end{proof}
Note that $(\ref{Lagrange})$ is related to Lagrange interpolation.
If we put $x=0$, then after simple transformations we obtain an identity that generalizes $(GF.14)$ and $(GF.15)$:
\beq\label{Lag2}
U_k^n=\sum_{i=0}^{n}\frac{U_{k+d_i}^n}{Q^{n d_i}}\prod_{\substack{j=0\\ j\not=i}}^{n}\frac{U_{d_j}}{U_{d_j-d_i}}.
\eeq
As an example, we give some Fibonacci identities that can be obtained from  $(\ref{Lagrange})$:
\begin{align}\label{FIF}
  F_{k+2}^3+F_{k-2}^3 & =3\(F_{k+1}^3-F_{k-1}^3\)+6F_k^3, \\
  F_{k+3}^4-F_{k-3}^4 & =4(F_{k+2}^4-F_{k-2}^4)+20(F_{k+1}^4-F_{k-1}^4), \\
  F_{k+3}^5-F_{k-3}^5 & =8(F_{k+2}^5+F_{k-2}^5)+40(F_{k+1}^5-F_{k-1}^5)-60F_{k}^5.
\end{align}
\begin{corollary}
Let $p_0$, $p_1$ be nonzeros  in $\mathbb{F}$, and let the sequence $\{W_k(a_0,a_1;p_0,p_1)\}$ be  defined by the relation
$W_{k+2}=p_0 W_{k+1}+p_1 W_{k}$, with $W_0=a_0, W_1=a_1$, where $a_0,a_1\in\mathbb{F}$. Let $s$ be an integer such that $W_s=0$, and $d_i$ $(0\leq i\leq n)$ be integers such that $W_{d_i-d_j+s}\not=0$ for $i\not= j$. Then the following holds
\beq \label{Lagrange1}
W_{k+x}^n=\sum_{i=0}^{n}W_{k+d_i}^n\prod_{\substack{j=0\\ j\not=i}}^{n}\frac{W_{x-d_j+s}}{W_{d_i-d_j+s}}.
\eeq
\end{corollary}
\begin{proof}
It is known \cite{1.11} that $W_k(a_0,a_1;p_0,p_1)=a_1U_k(p_0,-p_1)+a_0 p_1 U_{k-1}(p_0,-p_1)$.
The rest of the proof is analogous to the proof of Theorem 2.
\end{proof}
\smallskip
\noindent{\bf Acknowledgments.} I would like to thank K.A. Sveshnikov and R.M.  Kolpakov  for valuable suggestions. I am also very grateful to the referee for many helpful comments and suggestions. Finally, I want to thank  Max Alekseyev for pointing out the connection between  the identity $(\ref{Lag2})$ and Lagrange interpolation.

\medskip

\end{document}